\documentclass[12pt,letterpaper,reqno]{amsart}
\usepackage{graphicx}
\usepackage{amsmath}
\usepackage{amssymb}
\usepackage{amsfonts}
\usepackage{amsrefs}
\usepackage{enumerate}
\usepackage[mathscr]{eucal}

\usepackage[usenames]{color}

\newcommand{\R}{\mathbb{R}}
\newcommand{\cR}{\mathcal{R}}
\newcommand{\cD}{\mathcal{D}}
\newcommand{\Met}{\emph{Met}}

\newcommand{\ol}{\overline}

\newtheorem{thm}{Theorem}
\newtheorem{lemma}[thm]{Lemma}

\theoremstyle{definition}
\newtheorem{definition}[thm]{Definition}

\theoremstyle{remark}
\newtheorem{remark}{Remark}
\newtheorem*{ack}{Acknowledgments}

\DeclareMathOperator{\Div}{div}

\DeclareMathOperator{\Hess}{Hess}
\DeclareMathOperator{\tr}{tr}
\DeclareMathOperator{\dist}{dist}

\def\madm{m_{\mathrm{ADM}}}

\begin{document}

\title[Lower semicontinuity of the ADM mass]{Lower semicontinuity of the ADM mass in dimensions two through seven}
\author{Jeffrey L. Jauregui}
\address{Dept. of Mathematics,
Union College, 807 Union St.,
Schenectady, NY 12308}
\email{jaureguj@union.edu}

\begin{abstract}
The semicontinuity phenomenon of the ADM mass under pointed (i.e., local) convergence of asymptotically flat metrics is of interest because of its connections to nonnegative scalar curvature, the positive mass theorem, and Bartnik's mass-minimization problem in general relativity. In this paper, we
extend a previously known semicontinuity result in dimension three for $C^2$ pointed convergence to higher dimensions, up through seven, using recent work of S. McCormick and P. Miao (which itself builds on the Riemannian Penrose inequality of H. Bray and D. Lee).
For a technical reason, we restrict to the case in which the limit space is asymptotically Schwarzschild. In a separate result, we show that semicontinuity holds under weighted, rather than pointed, $C^2$ convergence, in all dimensions $n \geq 3$, with a simpler proof independent of the positive mass theorem. Finally, we also address the two-dimensional case for pointed convergence, in which the asymptotic cone angle assumes the role of the ADM mass.
\end{abstract}

\maketitle 

\section{Introduction}

Motivated by the Bartnik minimal mass extension conjecture in general relativity \cites{Ba1, Ba2, Ba3}, as well as the study of Ricci flow on asymptotically flat manifolds \cites{DM, OW}, in \cite{J} the author established the following result regarding how the ADM mass behaves under pointed convergence of a sequence of asymptotically flat 3-manifolds of nonnegative scalar curvature. Briefly, the ADM mass cannot increase in a local $C^2$ limit:

\begin{thm}[\cite{J}]
\label{thm1}
Let $(M_i, g_i, p_i)$ be a sequence of pointed asymptotically flat 3-manifolds without boundary, such that each $(M_i ,g_i)$ has nonnegative scalar curvature and contains no compact minimal surfaces. If $(M_i, g_i, p_i)$ converges in the pointed $C^2$ Cheeger--Gromov sense to a pointed asymptotically flat 3-manifold $(N, h, q)$, then 
\begin{equation}
\label{eqn_LSC}
 \madm(N, h) \le \liminf_{i\to\infty} \madm(M_i, g_i).
\end{equation}
\end{thm} 
We recall the relevant definitions in Section \ref{sec_background}; for now we note that pointed $C^k$ Cheeger--Gromov convergence essentially means $C^k$ convergence of the metric tensors on compact subsets, modulo diffeomorphisms. Examples are given in \cite{J} in which strictness holds in \eqref{eqn_LSC}.

Theorem \ref{thm1} is intimately connected to scalar curvature and to the positive mass theorem (PMT) \cites{SY,W}. In \cite{J} it was shown that (\ref{eqn_LSC}) can fail without assuming nonnegative scalar curvature (and the absence of compact minimal surfaces). Somewhat surprisingly, a simple blow-up example in \cite{J} shows that Theorem \ref{thm1} actually implies the PMT. However, to prove Theorem \ref{thm1}, either the PMT itself, or a stronger result, is required. The key estimate in the proof of Theorem \ref{thm1} was the lower bound
\begin{equation}
\label{eqn_HI}
m_{ADM} \geq m_H(\Sigma)
\end{equation}
of the ADM mass in terms of the Hawking mass of an outward-minimizing surface $\Sigma$, established by G. Huisken and T. Ilmanen \cite{HI}. Note that \eqref{eqn_HI} is well-known to imply the PMT.

Two major questions were left unsettled in \cite{J}. First, to what extent does this lower semicontinuity property of the ADM mass hold for weaker convergence than $C^2$? Subsequently the author and D. Lee proved in \cite{JL} that the theorem continues to hold if only pointed $C^0$ convergence is assumed. Second, does Theorem \ref{thm1} generalize to higher dimensions?  The primary concern of the present paper is to address the latter question.

Unfortunately, a bound directly analogous  to (\ref{eqn_HI})
is unknown beyond dimension three: Huisken--Ilmanen's proof in $n=3$ uses ``Geroch monotonicity'' of the Hawking mass, which crucially relies on the Gauss--Bonnet theorem in one  dimension lower. Generally, the missing link in establishing Theorem \ref{thm1} in higher dimensions has been a useful quantitative lower bound for the ADM mass in terms of the geometry of an outward-minimizing surface. Fortunately, a recent result of S. McCormick and P. Miao \cite{MM} provides such an estimate (see Theorem \ref{thm_MM} below) that is sufficient for our purposes. Their work uses the Riemannian Penrose inequality in higher dimensions, due to H. Bray and D. Lee \cite{BL} (which itself was a generalization of Bray's original proof in dimension three \cite{B}). Our main result is:

\begin{thm}
\label{thm2}
Theorem \ref{thm1} is true with ``3'' replaced by ``$n$\!'', where $3 \leq n \leq 7$, provided the limit $(N,h)$ is asymptotically Schwarzschild.
\end{thm}

Note that the Riemannian manifolds $(M_i,g_i)$ need not be asymptotically Schwarzschild even if their limit $(N,h)$ is.

The restriction in Theorem \ref{thm2} of $n \leq 7$ is primarily due to the fact that it is the highest dimension in which the Riemannian Penrose inequality is currently known. It was pointed out in \cite{BL} that even the positive mass theorem for $n \geq 8$ is insufficient to automatically extend the Riemannian Penrose inequality to $n \geq 8$. We strongly conjecture that the $n \leq 7$ restriction is unnecessary, and that the asymptotically Schwarzschild hypothesis can be replaced with asymptotic flatness --- see Remark \ref{rmk_HF}.

\begin{remark}
It is reasonable to attempt to extend Theorem \ref{thm1} to spin manifolds in higher dimensions using Witten's spinor technique in his proof of the PMT \cite{W}. However, as pointed out to the author by H. Bray, it is not clear how to make effective use of the hypothesis of no compact minimal surfaces in the spinor argument --- and it was shown in \cite{J} that  (\ref{eqn_LSC}) can  fail without this hypothesis.
\end{remark}

We also prove two other related results. Assuming weighted  $C^2$ convergence, we establish in lower semicontinuity  of the ADM mass in all dimensions $n \geq 3$ (Theorem \ref{thm_weighted} below). Weighted convergence gives global control on the asymptotics of the metrics, in contrast to pointed convergence. In this case, with a stronger hypothesis than in Theorems \ref{thm1} and \ref{thm2}, the absence of compact minimal surfaces is unnecessary and the proof is easier. However, the weighted result does not recover the positive mass theorem. Our proof is a generalization of a result of Y. Li, who addressed the problem for a convergent asymptotically flat Ricci flow \cite{Li}. The notation below is explained in Section \ref{sec_weighted}:

\begin{thm}
\label{thm_weighted}
Suppose a sequence $\{g^\ell\}_{\ell=1}^\infty$ converges to $g$ as asymptotically flat Riemannian metrics in \emph{$\Met_{-\tau}^2(M)$}, where $\tau> \frac{n-2}{2}$. 
Then
\begin{align*}
\lim_{\ell \to \infty} &\left(m_{ADM}(M,g^\ell) -  \frac{1}{2(n-1) \omega_{n-1}} \int_M R(g^\ell) dV_{g^\ell}\right)\\
&\qquad\qquad\qquad\qquad \qquad\qquad= m_{ADM}(M,g) - \frac{1}{2(n-1) \omega_{n-1}} \int_M R(g) dV_{g},
\end{align*}
where $dV_{g^\ell}$ and $dV_g$ are the volume measures of $g^\ell$ and $g$. Moreover, if $R(g^\ell) \geq 0$ for all $\ell$, then
$$\liminf_{\ell \to \infty} m_{ADM}(M,g_\ell) \geq m_{ADM}(M,g).$$ 
\end{thm}

Finally, it was suggested by E. Woolgar that the author investigate the lower semicontinuity of ``mass'' in dimension two. This is carried out in Section \ref{sec_2D} for pointed $C^2$ convergence, where the asymptotic cone angle replaces the ADM mass; see Theorem \ref{thm_cone}.

\begin{ack}
The author thanks J. Corvino, D. Lee, S. McCormick, and P. Miao for helpful discussions.
The author acknowledges support from the Erwin Schr\"odinger Institute, at which a portion of this work was completed in 2017.
\end{ack}

\section{Motivation and examples}
\label{sec_examples}

In this section we describe several examples to motivate the lower semicontinuity phenomenon for the ADM mass.

\subsection{Lower semicontinuity of mass in Newtonian gravity}
\label{sec_NG}
We begin here with a general discussion of why lower semicontinuity of the total mass is plausible from the point of view of Newtonian gravity. Consider a matter distribution on $\R^n$ described by a continuous, integrable mass density function $\rho \geq 0$. The total Newtonian mass is simply given by the integral:
$$m(\rho) = \int_{\R^n} \rho dx^1\ldots dx^n.$$
Now, if $\{\rho_i\}_{i=1}^\infty$ is a sequence of such matter distributions that converges pointwise to $\rho$, then by Fatou's lemma,
$$\liminf_{i \to \infty} m(\rho_i) \geq m(\rho).$$
Any drop in the total Newtonian mass can be viewed as mass escaping out to infinity in the limit. Such an argument does not apply to the context of general relativity, because the ADM mass is not known (or expected) to be given as the integral of a locally defined, nonnegative, geometric/physical quantity. 

Convergence of $\rho_i$ to $\rho$ in Newtonian gravity is analogous to $C^2$ convergence of the Riemannian metrics in general relativity, as the scalar curvature represents energy density and is given by two derivatives of the metric. The $C^0$ convergence in \cite{JL} can then be viewed as a general relativistic analog of convergence of the Newtonian gravitational potentials $u_i \to u$, where $\Delta u_i = 4\pi\rho_i$ and $\Delta u = 4\pi\rho$.

\subsection{Blow-up example}
\label{sec_blowup}
In \cite{J} the author gave the example of a fixed asymptotically flat $n$-manifold $(M,g)$ of nonnegative scalar curvature and considered the sequence of homothetic rescalings $\{(M,i^2 g, p)\}$ for $p \in M$ fixed and $i = 1, 2, \ldots$. This sequence converges in the pointed $C^2$ Cheeger--Gromov sense to Euclidean $\R^n$ (which has zero ADM mass), and indeed the statement of lower semicontinuity of mass implies that the ADM mass of $(M,g)$ is nonnegative. In other words, the positive mass theorem is recovered.

The example in section \ref{sec_NG} suggests that from a Newtonian point of view, the mass-drop phenomenon can be completely accounted for by matter escaping off to infinity. But by choosing $(M,g)$ here to be scalar-flat (i.e., vacuum) with positive ADM mass, the example of $\{(M,i^2 g, p)\}$ converging to Euclidean space shows that the mass can drop by an infinite amount in the limit with no matter fields present.  This can be interpreted as the energy of the gravitational field escaping to infinity.

\subsection{Escaping point example}
\label{sec_escape}
Similar to the previous example, begin with a fixed asymptotically flat $n$-manifold $(M,g)$. Now consider a sequence of points $\{p_i\}$ in $M$ escaping to infinity. By asymptotic flatness, the sequence $\{(M,g,p_i)\}$ converges in the pointed $C^2$ Cheeger--Gromov sense to Euclidean $\R^n$. Again the statement of lower semicontinuity of ADM mass here recovers the positive mass theorem; and again by choosing $(M,g)$ to be scalar-flat with positive ADM mass we can interpret the mass drop as gravitational energy escaping to infinity.

\subsection{Lower semicontinuity of mass and Ricci flow}

To the author's knowledge, the ADM mass drop phenomenon under pointed convergence was first observed by T. Oliynyk and E. Woolgar in their study of Ricci flow on rotationally symmetric, asymptotically flat spaces \cite{OW}; see also the work of X. Dai and L. Ma, who first showed that the ADM mass is constant along Ricci flow, thereby arguing an asymptotically flat Ricci flow cannot converge uniformly to Euclidean space \cite{DM}.  
Under natural hypotheses, Oliynyk--Woolgar proved the long-time existence of Ricci flow on asymptotically flat, rotationally symmetric spaces, with pointed $C^k$ Cheeger--Gromov convergence to Euclidean space as $t \to \infty$. Moreover, the ADM mass is not only monotone but is in fact constant along the Ricci flow. In particular, if the initial space has positive ADM mass, then the ADM mass must drop to zero in the limit.

In light of this discussion, the author suggested in \cite{J} that using Theorem \ref{thm1} (or its higher dimensional analog) would be necessary in any proof of the PMT that involved convergence of the Ricci flow to Euclidean space. Since Theorem \ref{thm1} already subsumes the PMT, this seemed to suggest than an independent Ricci flow proof of the PMT was unlikely. Nevertheless, such a proof has very recently been given by Y. Li in \cite{Li}. His argument circumvents this apparent circular logic by establishing lower semicontinuity of the ADM mass directly for the case of a convergent Ricci flow (i.e., the technique does not apply to general pointed $C^2$ Cheeger--Gromov convergence). We generalize Li's argument to weighted $C^2$ convergence in Section \ref{sec_weighted}.

\section{Background}
\label{sec_background}

We begin with the definition of an asymptotically flat manifold (with one end). Many slight variants appear in the literature; the version below is commonly used.

\begin{definition}
A smooth, connected Riemannian $n$-manifold $(M,g)$, with $n \geq 3$, possibly with compact boundary, is \textbf{asymptotically flat (AF)} if there exists a compact set $K \subset M$ and a diffeomorphism
$\Phi: M \setminus K \to \R^n \setminus B$, for a closed ball $B$, such that in the ``asymptotically flat'' coordinates $x=(x^1, \ldots, x^n)$ given by $\Phi$, we have
\begin{equation}
\label{decay}
g_{ij} = \delta_{ij} + O(|x|^{-\tau}), \qquad \partial_kg_{ij} =  O(|x|^{-\tau-1}), \qquad \partial_k\partial_\ell g_{ij} = O(|x|^{-\tau-2}),
\end{equation}
for some constant $\tau>\frac{n-2}{2}$ (the \textbf{order}), and the scalar curvature of $g$ is integrable.
(Indices $i,j,k,\ell$ above run from $1$ to $n$, and $\partial$ denotes partial differentiation in the coordinate chart.)
\end{definition}

For example, for a real number $m>0$, the Schwarzschild metric
$$g_{ij} = \left(1+ \frac{m}{2|x|^{n-2}}\right)^{\frac{4}{n-2}} \delta_{ij}$$
on $\R^n$ minus a ball about the origin is asymptotically flat of order $n-2$.

We will also need two classes of of asymptotically flat manifolds with more restricted asymptotics at infinity:
\begin{definition}
An asymptotically flat Riemannian $n$-manifold $(M,g)$ is \textbf{asymptotically Schwarzschild} if there exists an ``asymptotically Schwarzschild coordinate system'' $(x^1, \ldots, x^n)$ on $M \setminus K$, i.e.
\begin{equation}
\label{eqn_gh}
g_{ij} =  \left(1+ \frac{m}{2|x|^{n-2}}\right)^{\frac{4}{n-2}} \delta_{ij} + h_{ij},
\end{equation}
for some real constant $m$, where
\begin{equation}
\label{decay2}
h_{ij}=O(|x|^{1-n}), \qquad \partial_k h_{ij} =  O(|x|^{-n}), \qquad \partial_k\partial_\ell h_{ij} = O(|x|^{-n-1}).
\end{equation}
\end{definition}

Note that an asymptotically Schwarzschild Riemannian $n$-manifold is AF of order $n-2$.

\begin{definition}
\label{def_HF}
An asymptotically flat Riemannian $n$-manifold $(M,g)$ is \textbf{harmonically flat at infinity (HF)} if there exists a ``harmonically flat coordinate system'' $(x^1, \ldots, x^n)$ on $M \setminus K$, i.e.
$$g_{ij} = U^{\frac{4}{n-2}} \delta_{ij},$$
on $M \setminus K$ for some function $U$, where $\Delta U = 0$ and $U(x) \to 1$ as $|x| \to \infty$. (Here $\Delta$ is the Euclidean Laplacian on $\R^n$.)
\end{definition}

It is well-known that the harmonic function $U$ appearing in Definition \ref{def_HF} admits an expansion at infinity of the form:
\begin{equation}
\label{eqn_U}
U(x) = 1 + \frac{a}{|x|^{n-2}} + O_\infty (|x|^{-n+1}),
\end{equation}
where the notation $O_k(|x|^{\ell})$ denotes an expression that is $O(|x|^{\ell})$ for $|x|$ large and 
for which the $\gamma$th partial derivative ($\gamma$ being a multi-index with $|\gamma| \leq k$) is $O(|x|^{\ell-|\gamma|})$.
The fact that $\Delta U=0$ implies that $g$ as above has zero scalar curvature outside of $K$. Note that HF manifolds are necessarily asymptotically Schwarzschild, and that the Schwarzschild metric itself is HF.

Next, we recall the definition of ADM mass.
\begin{definition}
The \textbf{ADM mass} \cite{ADM} (cf. \cites{Ba0,Chr}) of an asymptotically flat manifold $(M,g)$ of dimension $n$ is the real number
$$m_{ADM}(M,g) = \frac{1}{2(n-1)\omega_{n-1}} \lim_{r \to \infty} \int_{S_r} \sum_{i,j=1}^n\left( \partial_ig_{ij} - \partial_j g_{ii}\right)\frac{x^j}{r} dA ,$$
where $dA$ is the induced volume form on the coordinate sphere $S_r=\{|x|=r\}$ with respect to the Riemannian metric $\delta_{ij}$, all in an AF coordinate chart.
\end{definition}


It is straightforward to verify that for an HF manifold, the ADM mass is given by the value $2a$, where $a$ is the constant appearing in (\ref{eqn_U}), and for an asymptotically Schwarzschild manifold, the ADM mass is given by the constant $m$ appearing in \eqref{decay2}.

\medskip

Recall that if $(M,g)$ is asymptotically flat with boundary $\partial M$, then we say $\partial M$ is \emph{outward-minimizing} if
$$|S| \geq |\partial M|$$
for all surfaces $S$ enclosing $\partial M$, where $|\cdot|$ denotes the hypersurface area (with respect to $g$). The following theorem was recently proved by McCormick and Miao \cite{MM}.

\begin{thm}[\cite{MM}]
\label{thm_MM}
Let $(M,g)$ be an AF manifold of dimension $3 \leq n \leq 7$, with compact, connected boundary $\Sigma$ that is outward-minimizing. Assume that the scalar curvature of $(M,g)$ is nonnegative. Let $H \geq 0$ be the mean curvature of $\Sigma$ (in the direction pointing into $M$), let $\rho$ be the scalar curvature of $\Sigma$ with respect to the induced Riemannian metric, and suppose that
$$\min_{\Sigma} \rho > \frac{n-2}{n-1} \max_\Sigma H^2.$$
Then
\begin{equation}
\label{eqn_MM}
m_{ADM}(M,g) \geq  \frac{1}{2} \left(\frac{|\Sigma|}{\omega_{n-1}}\right)^{\frac{n-2}{n-1}}\left(1 - \frac{n-2}{n-1}  \frac{\max_\Sigma H^2}{\min_\Sigma \rho}  \right).
\end{equation}
\end{thm}

To simplify notation later, we make the following definition.

\begin{definition}
\label{def_F_g}
Let $S$ be a smooth, compact hypersurface in a Riemannian manifold $(M,g)$ of dimension $n \geq 3$. Define
$$F_g(S) =  \frac{1}{2} \left(\frac{|S|}{\omega_{n-1}}\right)^{\frac{n-2}{n-1}}\left(1 - \frac{n-2}{n-1} \cdot  \frac{\max_S H^2}{\min_S \rho}  \right)$$
where $|S|$, $H$, and $\rho$ are the area, mean curvature, and scalar curvature of $\Sigma$ with respect to the Riemannian metric induced by $g$.
\end{definition}


We conclude this section with the definition of convergence used in Theorems \ref{thm1} and \ref{thm2}.
\begin{definition}
\label{def_CG}
Fix a nonnegative integer $\ell$. A sequence of complete, connected, pointed Riemannian $n$-manifolds $(M_i,g_i,p_i)$ converges in the \textbf{pointed $C^{\ell}$ Cheeger--Gromov sense} to a complete, connected, pointed Riemannian $n$-manifold $(N,h,q)$ if for every $r > 0$ there exists a domain $\Omega$ containing the metric ball $B_h(q,r)$ in $(N,h)$, and there exist (for all $i$ sufficiently large) smooth embeddings $\Phi_i: \Omega \to M_i$ such that $\Phi_i(\Omega)$ contains the metric ball $B_{g_i}(p_i,r)$, and the Riemannian metrics $\Phi_i^* g_i$ converge in $C^{\ell}$ norm to $h$ as tensors on $\Omega$.
\end{definition}

Note that no $M_i$ need be diffeomorphic to $N$ in the above definition, and that the asymptotics of $M_i$ can be wildly different from those of $N$ in the noncompact case.

\section{The mass of asymptotically Schwarzschild metrics}
\label{sec_HF}

In this section we prove that the ADM mass of an asymptotically Schwarzschild manifold can be recovered from the $r \to \infty$ limit of the expression $F_g(S_r)$, a key ingredient in the proof of Theorem \ref{thm2}. Before doing so (in Lemma \ref{lemma_ADM_AS}), we first verify this for HF metrics in Lemma \ref{lemma_ADM_HF}.

\begin{remark}
For an asymptotically flat manifold $(M,g)$ of dimension $3\leq n \leq 7$, the inequality
$$m_{ADM}(M,g) \geq \limsup_{r \to \infty} F_g(S_r)$$
follows from Theorem \ref{thm_MM}. However, equality need not hold. Such an example, pointed out to the author by S. McCormick, can be found by considering an AF manifold $(M,g)$ of nonnegative scalar curvature and strictly positive ADM mass that contains an isometric copy of half of a Euclidean space. Such spaces were constructed by Carlotto and Schoen \cite{CS}. For $r$ sufficiently large, $S_r$ intersects the Euclidean region in $M$, which gives $F_g(S_r) \leq 0$.
\end{remark}

\begin{lemma}
\label{lemma_ADM_HF}
If $(M,g)$ is an HF manifold, then
\begin{equation}
\label{eqn_F_g_limit}
m_{ADM}(M,g) = \lim_{r \to \infty} F_g(S_r),
\end{equation}
where $F_g$ is given in Definition \ref{def_F_g}, and $S_r$ is the coordinate sphere $\{|x|=r\}$ in a harmonically flat coordinate system.
\end{lemma}

Except for the calculations \eqref{F_g} at the end of the following proof, the proof of Lemma \ref{lemma_ADM_AS} will be independent of Lemma \ref{lemma_ADM_HF}.

\begin{proof}
The proof involves straightforward computations of the asymptotic behavior, for large $r$, of the area, mean curvature, and scalar curvature of $S_r$. Let $U$ be the harmonic function as in Definition \ref{def_HF}, with expansion (\ref{eqn_U}). 

First we compute the area of $S_r$:
\begin{align*}
|S_r|_{g} &= \int_{S_r} U^{\frac{2(n-1)}{n-2}} dA\\
&= \int_{S_r} \left(1 + \frac{2a(n-1)}{(n-2)r^{n-2}} + O(r^{1-n})\right) dA\\
&= \omega_{n-1} r^{n-1}\left(1+ \frac{2a(n-1)}{(n-2)r^{n-2}} \right) + O(1),
\end{align*}
where $dA$ is the area form on $S_r$ induced by $\delta$.
In particular,
\begin{align*}
\frac{1}{2} \left(\frac{|S_r|_{g}}{\omega_{n-1}}\right)^{\frac{n-2}{n-1}} &= \frac{1}{2} r^{n-2}\left(1+ \frac{2a}{r^{n-2}} \right) + O(r^{-1}).
\end{align*}

Second we compute the mean curvature. Recall that the mean curvature of $S_r$ with respect to $\delta_{ij}$ is $\frac{n-1}{r}$. From a well-known formula relating the mean curvatures of conformally related Riemannian metrics, letting $H_r$ represent the mean curvature of $S_r$ with respect to $g$, we have
\begin{align*}
H_r &= U^{-\frac{2}{n-2}} \cdot \frac{n-1}{r} + \frac{2(n-1)}{n-2} \cdot U^{-\frac{n}{n-2}}\cdot \nu(U)\\
&= \left(1 + \frac{a}{r^{n-2}} + O(r^{1-n})\right)^{-\frac{2}{n-2}} \cdot \frac{n-1}{r}\\
&\qquad + \frac{2(n-1)}{n-2} \left(1 + \frac{a}{r^{n-2}} + O(r^{1-n})\right)^{-\frac{n}{n-2}} \left( -\frac{a(n-2)}{r^{n-1}} + O(r^{-n})\right)\\
&= \left(1 - \frac{2a}{(n-2)r^{n-2}} + O(r^{1-n})\right)\cdot \frac{n-1}{r}\\
&\qquad + \frac{2(n-1)}{n-2} \left(1 - \frac{an}{(n-2)r^{n-2}} + O(r^{1-n})\right) \left( -\frac{a(n-2)}{r^{n-1}} + O(r^{-n})\right)\\
&=\frac{n-1}{r}  -\frac{2a(n-1)^2}{(n-2)r^{n-1}} +O(r^{-n}),
\end{align*}
where we used the fact that the $\delta$-unit normal $\nu$ to $S_r$ equals $\frac{\partial}{\partial r}$.
Thus,
\begin{align*}
H_r^2 &=\frac{(n-1)^2}{r^2} - \frac{4a(n-1)^3}{(n-2)r^n} + O(r^{-n-1}).
\end{align*}

Third, we compute the scalar curvature of $S_r$ with respect to $g|_{TS_r}$. Recall that if $g_2=e^{2\psi}g_1$ are conformally related Riemannian metrics on a manifold of dimension $n-1$, then their scalar curvatures are related by
$$R_{g_2} = e^{-2\psi}\left(R_{g_1} - 2(n-2) \Delta_{g_1} \psi - (n-3)(n-2)|d\psi|^2_{g_1}  \right).$$
In particular, with $g_2 = g|_{TS_r}$, $g_1=\delta|_{TS_r}$, and $U^{\frac{4}{n-2}} = e^{2\psi}$ on $S_r$, we have
\begin{align*}
\rho 
&= U^{-\frac{4}{n-2}}  \left(\frac{(n-1)(n-2)}{r^2} - \frac{4\Delta_r U}{U} + \frac{4|\nabla_r U|^2}{(n-2)U^2}\right),
\end{align*}
where $\Delta_r$ and $\nabla_r$ are the Laplacian and (tangential) gradient on $S_r$ with the Riemannian metric induced from $\delta$, and  $|\cdot|^2$ is taken with respect to $\delta$. Now, we address the Laplacian term. A well-known formula for smooth functions $f$ on $\R^n$ is 
$$\Delta f = \Delta_\Sigma f + \Hess(f)(\nu,\nu) + H \partial_\nu(f),$$
where $\Sigma$ is a smooth hypersurface with unit normal $\nu$, mean curvature $H$ in the direction of $\nu$, and induced Laplacian $\Delta_\Sigma$. Applying this to $f=U$ and $\Sigma = S_r$, we have
$$0 = \Delta_r U+\Hess(U)(\partial_r,\partial_r) + \frac{n-1}{r} \cdot \frac{\partial U}{\partial r}.$$
By explicit calculation, the leading (i.e., $O(r^{-n}))$ terms of $ \Hess(U)(\partial_r,\partial_r)$ and $\frac{n-1}{r} \cdot \frac{\partial U}{\partial r} $ cancel, implying that
$$\Delta_r U = O(r^{-n-1}).$$
Next, for the term $|\nabla_r U|$, since $1+\frac{a}{r^{n-2}}$ is constant on $S_r$, we see from the expansion of $U$ that
$$|\nabla_r U|^2 = O(r^{-2n}).$$
Using these expansions, along with the expansion for $U$, we arrive at
\begin{align*}
\rho &= \left(1+ \frac{a}{r^{n-2}} + O(r^{1-n})\right)^{-\frac{4}{n-2}} \left(\frac{(n-1)(n-2)}{r^2}  +O(r^{-n-1})\right)\\
&= \frac{(n-1)(n-2)}{r^2} - \frac{4a(n-1)}{r^n} + O(r^{-n-1}).
\end{align*}

Putting it all together, we have
\begin{align}
F_g(S_r) &= \left(\frac{1}{2} r^{n-2}\left(1+ \frac{2a}{r^{n-2}} \right) + O(r^{-1})\right) \left(1-\frac{n-2}{n-1}\cdot\frac{\frac{(n-1)^2}{r^2} - \frac{4a(n-1)^3}{(n-2)r^n} + O(r^{-n-1})}{\frac{(n-1)(n-2)}{r^2} - \frac{4a(n-1)}{r^n} + O(r^{-n-1})}\right) \nonumber\\
&= \left(\frac{1}{2} r^{n-2}+  a + O(r^{-1})\right) \left(1-\frac{1 - \frac{4a(n-1)}{(n-2)r^{n-2}} + O(r^{-n+1})}{1 - \frac{4a}{(n-2)r^{n-2}} + O(r^{-n+1})}\right)\nonumber\\
&=\left(\frac{1}{2} r^{n-2}+  a + O(r^{-1})\right) \left(\frac{4a}{r^{n-2}} + O(r^{-n+1})\right)\nonumber\\
&= 2a + O(r^{-1}). \label{F_g}
\end{align}
Since the ADM mass of $g$ equals $2a$, the proof is complete.
\end{proof}

The next lemma is a generalization of the previous:

\begin{lemma}
\label{lemma_ADM_AS}
If $(M,g)$ is an asymptotically Schwarzschild manifold, then
\begin{equation*}
m_{ADM}(M,g) = \lim_{r \to \infty} F_g(S_r),
\end{equation*}
where $S_r$ is the coordinate sphere $\{|x|=r\}$ in an asymptotically Schwarzschild coordinate system.
\end{lemma}

\begin{proof}
This follows from Lemma \ref{lemma_asymp} in the Appendix and the calculations \eqref{F_g} at the end of the proof of the previous lemma.
\end{proof}

\section{Proof of Theorem \ref{thm2}}
\label{sec_proof}

The method of proof of Theorem \ref{thm2} is similar to the proof of Theorem \ref{thm1} in \cite{J}. 

Let $m_i = m_{ADM}(M_i,g_i)$, and note that $m_i \geq 0$ by the positive mass theorem in dimension $3 \leq n \leq 7$ (\cites{SY,SY2}, cf. Section 4 of \cite{Sch}). If $m_{ADM}(N,h)=0$, the claim \eqref{eqn_LSC} follows trivially, so we may assume it is strictly positive. 

Let $\epsilon > 0$. Fix an asymptotically Schwarzschild coordinate  system $(x^1, \ldots, x^n)$ on $(N,h)$, and let $S_r$ denote the coordinate sphere $\{|x|=r\}$, a smooth, compact hypersurface in $N$ for $r$ sufficiently large. Let $B_r$ denote the bounded open region in $N$ that $S_r$ encloses.

By Lemma \ref{lemma_ADM_AS} and the hypothesis that $(N,h)$ is asymptotically Schwarzschild of positive ADM mass, we may choose  a number $r_1>0$ sufficiently large so that
\begin{align}
m_{ADM}(N,h)&<F_h(S_{r_1}) + \frac{\epsilon}{2}\label{eqn_F_h} \qquad\text{ and }\\
F_h(S_{r_1})&>0.\label{eqn_F_h2}
\end{align}
By asymptotic flatness of $h$, we may increase $r_1$ if necessary, preserving \eqref{eqn_F_h} and \eqref{eqn_F_h2}, to arrange that the mean curvature of $S_r$ with respect to $h$ is strictly positive for all $r \geq r_1$, and that hypersurface areas measured with respect to $h$ and the Euclidean metric $\delta$ differ by at most a factor of 2 on $N \setminus B_{r_1}$ (i.e., the respective Hausdorff $(n-1)$-measures are uniformly equivalent by factors of 2).

We apply the definition of pointed $C^2$ Cheeger--Gromov convergence. First, take a number $r_2>0$ so that the metric ball $B_h(q,r_2)$ contains $B_{33r_1}$. (The value $33r_1$ is chosen because later we will need a point in $B_{33r_1} \setminus B_{r_1}$ that is distance $16r_1$ from both the inner and outer boundary.) Then there exists a domain $U \subset N$, with $U \supset B_h(q,r_2) \supset B_{33r_1}$, and smooth embeddings $\Phi_i: U \to M_i$, for $i \geq$ some $i_0$, with $\Phi_i(U) \supset B_{g_i}(p_i,r_2)$, such that
\begin{equation}
\label{eqn_h_i}
h_i:=\Phi_i^* g_i \to h \text{ in } C^2 \text{ on } U.
\end{equation}
(Below, we will repeatedly 
use the fact that $\Phi_i : (U,h_i) \to (\Phi_i(U), g_i)$ is trivially an isometry.)
Taking $i$ to be at least some $i_1 \geq i_0$, we can be sure that  hypersurface areas measured with respect to $h_i$ and $h$ differ by at most a factor of 2 on $U$, by $C^0$ convergence. Taking $i$ to be at least some $i_2 \geq i_1$, we can arrange that the mean curvatures of $S_r$ with respect to $h_i$ are strictly positive for all $r \in [r_1, 33r_1]$, using $C^1$ convergence of $h_i$ to $h$ on $U$. 

Next, let $S_i = \Phi_i(S_{r_1})$, a smooth compact hypersurface in $M_i$. We want to apply Theorem \ref{thm_MM} to the AF manifold-with-boundary obtained by removing $\Phi_i(B_{r_1})$ from $M_i$ (whose boundary is $S_i$). To do so, we must verify that $S_i$ is outward-minimizing in $(M_i, g_i)$. (This is not at all obvious, since $S_i$ need not even lie in the asymptotically flat end of $(M_i,g_i)$.) This issue was handled in \cite{J} via a monotonicity formula for minimal surfaces in a Riemannian manifold. However, we will instead use the more robust argument in \cite{JL}, using the notion of almost-minimizing currents.

\begin{lemma}
\label{lemma_tentacle}
For $i \geq i_2$, $S_i$ is (strictly) outward-minimizing in $(M_i,g_i)$.
\end{lemma}

\begin{proof}[Proof of Lemma \ref{lemma_tentacle}]
It is well-known from standard results in geometric measure theory (see \cite{HI} for instance) that there exists a compact hypersurface $\tilde S_i$ enclosing $S_i$ that has the least hypersurface area (with respect to $g_i$) among all compact hypersurfaces in $M_i$ enclosing $S_i$.  Moreover, $\tilde S_i$ has at least $C^{1,1}$ regularity, and $\tilde S_i \setminus S_i$, if nonempty, is a smooth minimal hypersurface. (This uses $n \leq 7$.) We complete the proof of the lemma by arguing that $\tilde S_i = S_i$, assuming henceforth that $i \geq i_2$.

If $\tilde S_i$ were to possess a connected component disjoint from $S_i$, then that component would be a compact minimal hypersurface in $(M_i, g_i)$, contrary to the hypothesis of Theorem \ref{thm2}. Thus, every connected component of $\tilde S_i$ intersects $S_i$.

Next, if $\tilde S_i$ happens to be contained in the compact region $\Phi_i(\ol B_{33r_1})$ and hence in $\Phi_i(\ol B_{33r_1} \setminus B_{r_1})$, there exists some point $p \in \tilde 
S_i$ at which the function $r \circ \Phi^{-1}|_{\tilde S_i}$ achieves its maximum on $\tilde S_i$. Say this maximum value is $r^* \in 
[r_1, 33r_1]$. If $r^*> r_1$, then $\tilde S_i$ is smooth and minimal (with respect to $g_i$) near $p$ and is tangent to 
$\Phi_i(S_{r^*})$. However, this contradicts the standard comparison principle for mean curvature, 
as $\Phi_i(S_{r^*})$ has  strictly positive mean curvature with respect to $g_i$ (because $S_{r^*}$ has 
strictly positive mean curvature with respect to $h_i$). Thus, $r^*= r_1$, and so $\tilde S_i = S_i$, as claimed.

The only remaining case is that $\tilde S_i$ possesses a connected component, say $\tilde S_i'$, that is not contained in $\Phi_i(\ol B_{33r_1} \setminus B_{r_1})$, but that intersects $S_i = \Phi_i(S_{r_1})$.  Let
$$T_i= \Phi_i^{-1}(\tilde S_i' \cap \Phi_i( B_{33r_1} \setminus \bar B_{r_1})) \; \subset \; B_{33r_1} \setminus \bar B_{r_1} \;\subset\; N.$$ 
Note that $T_i$ is a smooth hypersurface in the AF end of $N$, so that we may regard $T_i \subset \R^n$ with $\partial T_i \subset S_{r_1} \cup S_{33r_1}$.
By the connectedness of $\tilde S_i'$ and the continuity of $r$, there exists some point $q_i \in T_i \cap S_{17r_1}$, and the Euclidean distance from $q_i$ to $\partial T_i$ is $16r_1$.
Viewing $T_i$ naturally as an $(n-1)$-dimensional integral current in $\R^n$, we claim that $T_i$ is $\gamma$-\emph{almost-minimizing} for $\gamma=16$ (and will verify this later). Recall this means that given any ball $B$ in $\R^n$ that does not intersect $\partial T_i$, and any integral current $T$ with the same boundary as the restriction $T_i \llcorner B$, we have
$$|T_i \llcorner B|_\delta \leq \gamma |T|_\delta$$
for some constant $\gamma \geq 1$. (Here we are using $|\cdot|_\delta$ to denote both the Euclidean hypersurface area and the more general current mass.)
The following fact is a natural generalization of the classical monotonicity formula for minimal surfaces to the class of $\gamma$-almost-minimizing currents (see \cite{BL} for instance):
for $0 \leq s < \dist(q_i, \partial T_i) = 16r_1$, 
$$|T_i \llcorner B(q_i,s)|_\delta \geq \gamma^{2-n} \omega_{n-1} s^{n-1}.$$
Taking the limit $s \nearrow 16r_1$, we have
$$|T_i \llcorner B(q_i,16r_1)|_\delta \geq \gamma^{2-n} \omega_{n-1} (16r_1)^{n-1}=16 \omega_{n-1} (r_1)^{n-1},$$
taking $\gamma=16$. Using the factor-of-two area comparisons between $\delta$ and $h$ and between $h$ and $h_i$ on $U \setminus B_{r_1}$ for $i\geq i_2$, we then have
$$|T_i \llcorner B(q_i,16r_1)|_{h_i} \geq \frac{1}{4}\cdot 16 \omega_{n-1} (r_1)^{n-1}.$$
Applying $\Phi_i$,
it follows that $|\tilde S_i' \cap \Phi_i(\ol B_{33r_1})|_{g_i} \geq 4 \omega_{n-1} (r_1)^{n-1}$.  Since $\tilde S_i$ leaves $\Phi_i(\ol B_{33r_1})$, we obtain a strict inequality below:
\begin{equation}
\label{eqn_area_estimate}
|\tilde S_i|_{g_i} \geq |\tilde S_i'|_{g_i} > 4\omega_{n-1} (r_1)^{n-1}.
\end{equation}

On the other hand, since $\tilde S_i$ by definition has at most as much $g_i$-area as $S_i$,
$$|\tilde S_i|_{g_i} \leq |S_i|_{g_i} = |S_{r_1}|_{h_i} \leq 4|S_{r_1}|_{\delta} = 4 \omega_{n-1} (r_1)^{n-1},$$
producing a contradiction with \eqref{eqn_area_estimate}.

We now prove that  $T_i$ is $\gamma$-almost-minimizing in $\R^n$ with $\gamma=16$, which will complete the proof of Lemma \ref{lemma_tentacle}. 
Since $T_i$ is area-minimizing with respect to $h_i$ in  $B_{33r_1} \setminus \ol B_{r_1}$,
we know that 
$$|T_i \llcorner B|_{h_i} \leq |T|_{h_i},$$ 
for any integral current $T$ supported in $B_{33r_1} \setminus \ol B_{r_1}$, with $\partial T = \partial (T_i \llcorner B)$, where $B$ is a Euclidean ball in $B_{33r_1} \setminus \ol B_{r_1}$. For $i \geq i_2$, since the Hausdorff $(n-1)$-measures of $h$ and $h_i$ are uniformly equivalent by factors of two on $U$, this implies
$$|T_i \llcorner B|_{h} \leq 4|T|_{h}.$$
for such $B$ and $T$. Since $T_i$ is contained outside $S_{r_1}$, we can use the comparison of areas between $\delta$ and $h$ to see that
\begin{equation*}
|T_i \llcorner B|_{\delta} \leq 16|T|_{\delta}.
\end{equation*}
for such $B$ and $T$. However, in the definition of $\gamma$-almost-minimizing, one may without loss of generality consider competitors $T$ supported in $\bar B$, since $\bar B$ is convex. It follows that $T_i$ is 16-almost-minimizing, and the proof of Lemma \ref{lemma_tentacle} is complete.
\end{proof}

\medskip

We continue with the proof of Theorem \ref{thm2}. Observe that
$F_g(S)$ varies continuously with respect to $C^2$ perturbations of $g$ on any neighborhood of $S$, since the area, mean curvature, and scalar curvature depend continuously on $g$ and its first and second derivatives. Then by the $C^2$ convergence in (\ref{eqn_h_i}), we may restrict to $i$ at least as large as some $i_3\geq i_2$ so that
\begin{equation}
\label{eqn_C2_close}
F_{h}(S_{r_1}) \leq F_{h_i}(S_{r_1})+\frac{\epsilon}{2}
\end{equation}
and that 
\begin{equation}
\label{eqn_F_h_i_S_r}
F_{h_i}(S_{r_1})>0
\end{equation}
(since $F_h(S_{r_1})>0$ by \eqref{eqn_F_h2}). Lemma \ref{lemma_tentacle} and \eqref{eqn_F_h_i_S_r} show that  Theorem \ref{thm_MM} may be applied to $M_i$ minus the open region $\Phi_i(B_{r_1})$, which has (connected) boundary $S_i$. Thus:
\begin{equation}
\label{eqn_F_h_i}
F_{h_i}(S_{r_1}) =F_{g_i}(S_i) \leq m_i.
\end{equation}

Then for all $i \geq i_3$, we may combine (\ref{eqn_F_h}), (\ref{eqn_C2_close}), and (\ref{eqn_F_h_i}) to arrive at
$$m_{ADM}(N,h) < m_i + \epsilon.$$
Now, taking $\liminf_{i \to \infty}$ proves Theorem  \ref{thm2}, since $\epsilon>0$ was arbitrary.

\medskip

\begin{remark}
\label{rmk_higher_dim_C0}
The above proof generalizes the $C^2$ lower semicontinuity result from $n=3$ in \cite{J} to $3 \leq n \leq 7$. By contrast, extending the $C^0$ lower semicontinuity result in \cite{JL} to higher dimensions would be much more difficult. In the $C^0$ case, the dimension three hypothesis is relied on to a greater extent. First, the Hawking mass estimate \eqref{eqn_HI} of Huisken--Ilmanen, valid only in dimension three, is used to ensure monotonicity under mean curvature flow of a certain quantity (whose details we omit here) defined by Huisken. The author is not aware of such a monotone quantity in higher dimensions. Second, in \cite{JL}, use is made of B. White's regularity theory for the weak (level set) version of mean curvature flow that is especially nice in ambient dimension three \cite{Wh}.
\end{remark}

\begin{remark}
\label{rmk_HF}
As mentioned in the introduction, we strongly conjecture that the hypothesis that the limit $(N,h)$ is asymptotically Schwarzschild in Theorem \ref{thm2} (as opposed to asymptotically flat) is unnecessary. We note this generalization would follow by establishing a density result of the following form: Given $\epsilon>0$ and a sequence $(M_i,g_i,p_i)$ of AF manifolds of nonnegative scalar curvature converging in the pointed $C^2$ Cheeger--Gromov sense to an AF manifold $(N,h,q)$, construct a HF perturbation $\bar h$ of $h$  (with $|m_{ADM}(N,\bar h) - m_{ADM}(N, h)|<\epsilon$) and AF metrics $\bar g_i$ on $M_i$ of nonnegative scalar curvature, with $|m_{ADM}(M_i,\bar g_i) - m_{ADM}(M_i, g_i)|<\epsilon$, such that $(M_i, \bar g_i,p_i) \to (N, \bar h,q)$ in the pointed $C^2$ Cheeger--Gromov sense. Such a result would immediately generalize Theorem \ref{thm2} to remove the restriction that $(N,h)$ is asymptotically Schwarzschild, since HF manifolds are such.
\end{remark}

\section{Lower semicontinuity for weighted $C^2$ convergence in all dimensions}
\label{sec_weighted}

In this section we study the behavior of the ADM mass under \emph{weighted} $C^2$ convergence. This corresponds to a finer topology than that of pointed $C^2$ Cheeger--Gromov convergence. In particular it is easier here to establish semicontinuity of the ADM mass and to obtain a stronger result: Theorem \ref{thm_weighted} from the introduction is valid in all dimensions $n \geq 3$, requires no hypothesis on minimal surfaces, and does not rely on (nor recover) the positive mass theorem.

To describe the setup, let $M$ be a smooth $n$-manifold that admits an AF metric. Fix a compact set $K \subset M$ and an AF coordinate system on $M \setminus K$ (for some AF metric).
For an integer $k \geq 0$ and a real number $\tau>0$, let $C^k_{-\tau}(M\setminus K)$ denote the class of $C^k$ functions $f:M\setminus K \to \R$ for which the quantity
$$\|f\|_{C^k_{-\tau}(M \setminus K)} = \sum_{0\leq |\gamma| \leq k} \sup_{x \in M\setminus K} |x|^{|\gamma|+\tau} |\partial^\gamma f(x)|$$
is finite, 
where the partial derivatives are taken with respect to the coordinate chart, and $\gamma$ represents multi-indices.
Thus,  functions in $C^k_{-\tau}(M\setminus K)$ decay as $O(r^{-\tau})$ or faster as $r \to \infty$, with successively faster decay up through $k$th-order derivatives. Define $C^k_{-\tau}(M)$ to be the set of $C^k$ functions $f:M \to \R$ with $f|_{M \setminus K} \in C^k(M \setminus K)$, equipped with the norm given as the sum of $\|f\|_{C^k_{-\tau}(M \setminus K)}$ and the $C^k$ norm of $f|_K$.

Note that if $g$ is an AF metric on $g$ of order $\tau$ obeying the decay conditions \eqref{decay} in the fixed coordinate chart, then 
\begin{equation}
\label{eqn_gij}
g_{ij} - \delta_{ij} \in C^2_{-\tau}(M\setminus K).
\end{equation}
For $k\geq 2$ and $\tau >0$, we let $\Met^k_{-\tau}(M)$ denote the set of $C^k$ Riemannian metrics $g$ on $M$ satisfying \eqref{eqn_gij} in the fixed coordinate chart. (The ADM mass of $g \in \Met^k_{-\tau}(M)$ is well-defined if $\tau > \frac{n-2}{2}$ and the scalar curvature of $g$ is integrable \cites{Ba0,Chr}.)
We say a sequence of Riemannian metrics $\{g^\ell\}_{\ell=1}^\infty$ in $\Met^k_{-\tau}(M)$ converges to $g \in \Met^k_{-\tau}(M)$ as $\ell \to \infty$ if $\|g_{ij}^\ell - g_{ij}\|_{C^k_{-\tau}(M \setminus K)} \to 0$ for all $i$ and $j$ and the tensors $g^\ell|_K$ converge in $C^k$ to $g|_K$ as $\ell \to \infty$. 

For the reader's convenience, we restate Theorem \ref{thm_weighted} from the introduction.
\begin{thm}
\label{thm_weighted_restate}
Suppose $\{g^\ell\}_{\ell=1}^\infty$ converges to $g$ as asymptotically flat Riemannian metrics in \emph{$\Met_{-\tau}^2(M)$}, where $\tau> \frac{n-2}{2}$. 
Then
\begin{align}
\lim_{\ell \to \infty} &\left(m_{ADM}(M,g^\ell) -  \frac{1}{2(n-1) \omega_{n-1}} \int_M R(g^\ell) dV_{g^\ell}\right)\nonumber\\
&\qquad\qquad\qquad\qquad \qquad\qquad= m_{ADM}(M,g) - \frac{1}{2(n-1) \omega_{n-1}} \int_M R(g) dV_{g},\label{eqn_lim}
\end{align}
where $dV_{g^\ell}$ and $dV_g$ are the volume measures of $g^\ell$ and $g$. Moreover, if $R(g^\ell) \geq 0$ for all $\ell$, then
$$\liminf_{\ell \to \infty} m_{ADM}(M,g_\ell) \geq m_{ADM}(M,g).$$ 
\end{thm}
Our proof below is a generalization of that of Y. Li \cite{Li}, who studied the behavior of the ADM mass and integral of scalar curvature in the case of a convergent Ricci flow.
\begin{proof}
Let $g_0$ be a background Riemannian metric on $M$ whose expression in $M \setminus K$ in the given AF coordinate chart is $\delta_{ij}$. Let $\Div_0$ be the divergence operator on tensors  and $\Delta_0$ the Laplacian on functions with respect to $g_0$. Define the continuous operator $\cD: \Met^2_{-\tau}(M) \to C^0_{-\tau-2}(M)$ by
$$\cD(g) = \Div_0(\Div_0 g) - \Delta_0 \left( \tr_{g_0}(g)\right).$$
The significance of $\cD$ is the formula for the ADM mass of $g \in \Met^2_{-\tau}(M)$ (provided $\tau > \frac{n-2}{2}$ and the scalar curvature of $g$ is integrable):
\begin{equation}
\label{eqn_int_R}
m_{ADM}(g) = \frac{1}{2(n-1)\omega_{n-1}} \int_{M} \cD(g) dV_0,
\end{equation}
which follows immediately  from the divergence theorem. Here, $dV_0$ is the volume measure of $g_0$.

By the $\Met^2_{-\tau}(M)$ convergence of $g^\ell$ to $g$, we have $\cD(g^\ell) \to \cD(g)$ in $C^0_{-\tau-2}(M)$. However, since $\tau+2$ is generally less than the $O(r^{-n})$ threshold for integrability, we cannot immediately apply the dominated convergence theorem. (And since we have no control on the sign of $\cD(g^\ell)$, we cannot apply Fatou's lemma.)

We proceed instead by considering the difference between $\cD(\cdot)$ and $R(\cdot)$ (a well-known trick), where $R: \Met^2_{-\tau}(M) \to C^0_{-\tau-2}(M)$ is the scalar curvature operator. Working in the fixed chart on $M \setminus K$, for any Riemannian metric $h \in \Met^2_{-\tau}(M)$ with Christoffel symbols $\Gamma_{ij}^k$, we have
\begin{align*}
\cD(h) &= \partial_i \partial_j h_{ij} - \partial_j \partial_j h_{ii}\\
\cR(h) &= h^{jk}\left(\partial_{i} \Gamma^i_{jk} - \partial_{k} \Gamma^i_{i j} + \Gamma^m_{jk}\Gamma^i_{im} - \Gamma^{m}_{i j} \Gamma^i_{km}\right).
\end{align*}
By direct computation, $\cD(g^\ell) - R(g^\ell)$ is 
$O(r^{-2-2\tau})$, where $O(r^{-2-2\tau})$ here is uniform in 
$\ell$ and moreover goes to zero in $C^0_{-2-2\tau}(M)$ as 
$\ell \to \infty$. Since $2+2\tau > n$, this $O(r^{-2-2\tau})$ error term is 
uniformly bounded by an integrable function on $M$. Then by the 
dominated convergence theorem and pointwise convergence of 
$\cD(g^\ell) - R(g^\ell)$ to $\cD(g) - R(g)$,
$$\lim_{\ell \to \infty} \int_M \left(\cD(g^\ell) - R(g^\ell)
\right) dV_{g^\ell} = \int _M\left(\cD(g) - R(g)\right) dV.$$
Together with \eqref{eqn_int_R}, this proves \eqref{eqn_lim}.

For the last claim assume $R(g^\ell)\geq 0$ for all $\ell$, and let $\mu = \liminf_{\ell \to \infty} m_{ADM}(M,g_\ell)$. If $\mu = +\infty$, the claim follows trivially. Suppose $\mu$ is finite. Pass to a subsequence $\{(M, g^{\ell(k)})\}_{k}$ for which
$$\lim_{k \to \infty}  m_{ADM}(M, g^{\ell(k)}) = \mu.$$
By the first part of the theorem, the sequence $ \int_M R(g^{\ell(k)}) dV_{g^{\ell(k)}}$ then converges, and moreover
\begin{equation}
\label{eqn_mu}
\mu = m_{ADM}(M,g) +  \frac{1}{2(n-1) \omega_{n-1}} \left(\lim_{k \to \infty} \int_M R(g^{\ell(k)}) dV_{g^{\ell(k)}} -\int_M R(g) dV_{g}\right).
\end{equation}
By the (weighted) $C^2$ convergence of $g^{\ell(k)}$ to $g$ as $k \to \infty$,  we have pointwise convergence of the scalar curvatures and volume forms. Thus, by Fatou's lemma and the hypothesis $R(g^{\ell(k)})\geq 0$, the expression in parentheses in \eqref{eqn_mu} is nonnegative. This completes the proof if $\mu$ is finite. 

Finally, suppose $\mu = -\infty$; this would be precluded by the positive mass theorem, but we want the proof independent of the PMT. Then by \eqref{eqn_lim}, a subsequence $\{(M, g^{\ell(k)})\}_{k}$ has its integral of scalar curvature converging to $-\infty$. This violates the hypothesis $R(g^{\ell(k)}) \geq 0$.
\end{proof}

\begin{remark}
Interestingly, Theorem \ref{thm_weighted_restate} implies that for the case of weighted $C^2$ convergence, the mass drop is accounted for completely by the total matter (i.e., the integral of scalar curvature) escaping off to infinity, much like in the example in section \ref{sec_NG} from Newtonian gravity. This contrasts with the case of pointed $C^2$ Cheeger--Gromov convergence, in which the ADM mass can drop within the class of scalar-flat metrics (e.g., the examples in sections \ref{sec_blowup} or \ref{sec_escape}, choosing $(M,g)$  to be scalar-flat with positive ADM mass).
\end{remark}

\begin{remark}
Note that the lower semicontinuity of the ADM mass with respect to weighted $C^2$ convergence does not imply the positive mass theorem as in the blow-up example or escaping point example with pointed convergence in Section \ref{sec_examples}. In those cases, the metrics do not converge to Euclidean space in a weighted sense.
\end{remark}

\subsection{Example: mass drop with weighted convergence}
We conclude this section by describing an example of AF metrics $g_i$, with nonnegative scalar curvature, converging in $\Met^2_{-\tau}(M)$ with $\tau > \frac{n-2}{2}$ for which the ADM mass drops. Physically, the construction involves considering a sequence of shells of matter, of fixed total mass, at progressively larger radii. For $n \geq 3$, let $\rho: \R^n \to \R$ be a smooth, radially symmetric, nonnegative function supported in the annulus between radii $\frac{1}{2}$ and $1$, with $\int_{\R^n} \rho = 1$. For $i=1,2,\ldots$, define a sequence of smooth functions
$$\rho_i(x) = i^{-n}\rho(x/i),$$
which also satisfy $\int_{\R^n} \rho_i = 1$ and are supported in the annulus between radii $\frac{i}{2}$ and $i$.

By elliptic PDE theory (or ODE theory), there exists a unique smooth solution (for each $i$) to the linear elliptic problem:
$$\begin{cases}
-\Delta v_i = \rho_i & \text{ on } \R^n\\
v_i \to 0 & \text{ at infinity.}
\end{cases}$$
Recognizing $v_i(x) = i^{2-n} v_1(x/i)$, it is easy to see that $v_i \to 0$ in $C^2_{-\tau}(M)$ for any $\tau < n-2$ as $i \to \infty$. Fix $\tau \in \left( \frac{n-2}{2}, n-2\right)$.

For $i$ sufficiently large, $u_i:=1+v_i$ is positive, and the Riemannian metric $g_i:=u_i^{\frac{4}{n-2}} \delta$ is asymptotically flat. Note that the scalar curvature of $g_i$
$$R_i = -\frac{4(n-1)}{n-2} u_i^{-\frac{n+2}{n-2}} \Delta u_i =  \frac{4(n-1)}{n-2} u_i^{-\frac{n+2}{n-2}} \rho_i,$$
is integrable because it has compact support.

Now, $g_i$ converges to the Euclidean metric in $\Met^2_{-\tau}(M)$ as $i \to \infty$, and each $g_i$ has nonnegative scalar curvature. 
We show now (using the divergence theorem) that the ADM mass of $g_i$ is a positive constant, independent of $i$:
\begin{align*}
m_{ADM}(g_i) &= -\frac{2}{(n-2)\omega_{n-1}} \lim_{r \to \infty} \int_{S_r} \nu(u_i) dA\\
&= -\frac{2}{(n-2)\omega_{n-1}} \int_{\R^n} \Delta u_i dV\\
&= \frac{2}{(n-2)\omega_{n-1}} \int_{\R^n} \rho_i dV\\
&= \frac{2}{(n-2)\omega_{n-1}},
\end{align*}
where $dA$ and $dV$ are the hypersurface area and the volume forms with respect to the Euclidean metric.
However, the ADM mass of the limit, Euclidean $\R^n$, vanishes.

\section{Two-dimensional case of semicontinuity of mass}
\label{sec_2D}
In two dimensions, a natural replacement for asymptotically flat manifolds is the class of asymptotically conical surfaces, with the asymptotic cone angle playing the role of mass. The author thanks E. Woolgar for his suggestion to investigate the semicontinuity of mass in this setting.

Following \cite{IMS}, for $\alpha > 0$, let
$$g_\alpha = dr^2 + \alpha^2 r^2 d\theta^2,$$
a smooth Riemannian metric on $\R^2 \setminus \{0\}$ describing a cone. Note that $g_\alpha$ has vanishing Gauss curvature. Define a connected two-dimensional Riemannian manifold $(M,g)$ to  be \emph{asymptotically conical  with cone angle $2\pi\alpha >0$} if there exists a compact set $C \subset M$ such that $M \setminus C$ is diffeomorphic to the complement of a closed ball in $\R^2$, on which $g-g_\alpha = O_2(r^{-\tau})$ for some constant $\tau > 0$. In particular, the Gauss curvature of $g$ is $O(r^{-2-\tau})$ and hence integrable.

We recall here that the integral of the Gauss curvature captures the cone angle. To see this, let $B_r$ be the compact region bounded by the coordinate circle $\Gamma_r$ in $(M,g)$ for $r$ large. By the Gauss--Bonnet formula,
\begin{equation}
\label{GB}
\int_{B_r} K dA = 2\pi \chi(B_r) - \int_{\Gamma_r} \kappa_g ds,
\end{equation}
where $\kappa_g$ is the geodesic curvature of $\Gamma_r$ with respect to $g$. By the $O_2(r^{-\tau})$ decay of $g$ to $g_\alpha$, we have
$$\lim_{r \to \infty} \int_{\Gamma_r} \kappa_g ds = \lim_{r \to \infty} \int_{\Gamma_r} \kappa_{g_\alpha} ds = 2\pi\alpha,$$
the latter equality given by direct calculation, where $\kappa_{g_\alpha}$ is the geodesic curvature of $\Gamma_r$ with respect to $g_\alpha$.
Taking the limit $r \to \infty$ in \eqref{GB} (and noting that $\chi(B_r)$ is eventually a constant, $\chi(M)$), we have
\begin{equation}
\label{eqn_KdA0}
\int_M K dA = 2\pi(\chi(M) -1) + 2\pi (1-\alpha).
\end{equation}
Note that if $K \geq 0$, it follows that $\chi(M) > 0$, and using the fact that $M$ is topologically the connect sum of $\R^2$ and a compact, connected surface, it follows that $\chi(M)=1$ and that $M$ itself is topologically $\R^2$. 

We define the mass of an asymptotically conical surface to be:
$$m_{cone}(M,g) = 1-\alpha,$$
which we note is a dimensionless quantity. The positive mass theorem is then immediate: $K \geq 0$ implies $m_{cone} \geq 0$, and equality holds if and only if $K \equiv 0$ and $M$ is homeomorphic to $\R^2$, which holds if and only if $(M,g)$ is isometric to the Euclidean plane.

Below is the statement of $C^2$ pointed lower semicontinuity of the mass in two dimensions (i.e., upper semicontinuity of the cone angle). Note that no hypothesis on closed geodesics (the analogs of compact minimal hypersurfaces) is necessary.

\begin{thm}
\label{thm_cone}
Suppose $(M_i, g_i, p_i)$ converges in the pointed $C^2$ Cheeger--Gromov sense to $(N,h,q)$ as pointed asymptotically conical Riemannian 2-manifolds. Suppose each $(M_i, g_i)$ has nonnegative Gauss curvature. Then
\begin{equation}
\label{eqn_LSC_2D}
m_{cone}(N,h) \leq \liminf_{i \to \infty} m_{cone}(M_i,g_i).
\end{equation}
\end{thm}
An example for which strict inequality holds in \eqref{eqn_LSC_2D} can be found using the blow-up or escaping point examples in sections \ref{sec_blowup} and \ref{sec_escape}, beginning with an asymptotically conical surface with nonnegative Gauss curvature and $\alpha < 1$.

\begin{proof}
Let $\epsilon>0$. By the $C^2$ convergence, $h$ itself has nonnegative Gauss curvature $K_h$, so in particular $\chi(N)=1$. Then by \eqref{eqn_KdA0},
$$m_{cone}(N,h) = \frac{1}{2\pi} \int_N K_h dA_h.$$
Since $K_h$ is integrable, we may choose $r>0$ sufficiently large so that the coordinate ball $B_r \subset N$ satisfies
$$m_{cone}(N,h) <  \frac{1}{2\pi} \int_{B_r} K_h dA_h + \frac{\epsilon}{2}.$$
Choosing $U \supset B_r$ and obtaining appropriate embeddings $\Phi_i: U \to M_i$ such that $h_i:=\Phi_i^* g_i$ converges in $C^2$ to $h$, we may take $i$ sufficiently large so that
\begin{equation}
\label{eqn_KdA}
\frac{1}{2\pi}\int_{B_r} K_h dA_h -  \frac{\epsilon}{2} < \frac{1}{2\pi}\int_{B_r} K_{h_i} dA_{h_i} =  \frac{1}{2\pi}\int_{\Phi_i(B_r)} K_{g_i} dA_{g_i}.
\end{equation}
Since $(M_i,g_i)$ has nonnegative Gauss curvature, the right-hand side in \eqref{eqn_KdA} is an underestimate for $m_{cone}(M_i,g_i)$. Thus,
$$m_{cone}(N,h) < m_{cone}(M_i,g_i) + \epsilon$$
for $i$ sufficiently large. From this, the result follows.
\end{proof}

We leave it as an open problem to study the behavior of the cone angle under weaker forms of convergence, such as pointed $C^0$ Cheeger--Gromov, pointed Gromov--Hausdorff, or pointed Sormani--Wenger intrinsic flat convergence \cite{SW}.

\section*{Appendix: geometry of asymptotically Schwarzschild metrics}
The purpose of this appendix is to prove the following asymptotic estimates  for large coordinate spheres in an asymptotically Schwarzschild manifold. These were used in the proof of Lemma \ref{lemma_ADM_AS}.

\begin{lemma}
\label{lemma_asymp}
Let $(M, \tilde g)$ be an asymptotically Schwarzschild manifold of dimension $n \geq 3$ and ADM mass $m$. 
Let $S_r$ be the coordinate sphere of large radius $r$ in $M$. 
Let  $\tilde \rho$ be the scalar curvature of $S_r$ with respect to the metric induced from $\tilde g$, and let $\tilde H$ be the mean curvature of $S_r$ with respect to $\tilde g$. Then:
\begin{align}
\tilde \rho &= \frac{(n-1)(n-2)}{r^2} -\frac{2(n-1)m}{r^n} + O(r^{-n-1}) \label{eqn_tilde_rho}\\
\tilde H  &=\frac{n-1}{r}  -\frac{(n-1)^2m}{(n-2)r^{n-1}} +O(r^{-n}).  \label{eqn_tilde_H}
\end{align}

\end{lemma}

\begin{proof}
Let $g$ be the Schwarzschild metric of mass $m$, and let $h$ be as in \eqref{decay2}, i.e.
\begin{align*}
g &=  \left(1+ \frac{m}{2r^{n-2}}\right)^{\frac{4}{n-2}} \delta\\
\tilde g &= g + h,
\end{align*}
in the end of $M$.

We first address the scalar curvature of $S_r$. Let $\gamma$ and $\tilde \gamma$ be the Riemannian metrics on $S_r$ induced by $g$ and $\tilde g$, respectively. The coordinate sphere $S_r$ has constant scalar curvature with respect to the metric induced by $\delta$ equal to $\frac{(n-1)(n-2)}{r^2}$. Since the conformal factor relating $g$ to $\delta$ is constant on $S_r$, the scalar curvature of $(S_r,\gamma)$ can be found by rescaling:
\begin{align}
\rho &= \left(1+ \frac{m}{2r^{n-2}}\right)^{-\frac{4}{n-2}} \cdot \frac{(n-1)(n-2)}{r^2} \nonumber\\
&= \frac{(n-1)(n-2)}{r^2} -\frac{2(n-1)m}{r^n} + O(r^{-2(n-1)}). \label{eqn_rho}
\end{align}
We proceed to estimate the scalar curvature of $(S_r, \tilde \gamma)$ as follows.
Introduce spherical coordinates $(r,\phi^1, \ldots, \phi^{n-1})$ on the asymptotically flat end of $M$:
\begin{align*}
x^1 &= r \cos(\phi^1)\\
x^2 &= r \sin(\phi^1) \cos (\phi^2)\\
&\ldots\\
x^{n-1} &= r \sin(\phi^1) \sin(\phi^2) \ldots \cos (\phi^{n-1})\\
x^n &=r \sin(\phi^1) \sin(\phi^2) \ldots \sin (\phi^{n-1}).
\end{align*}
We use Greek indices for the directions tangent to $S_r$, i.e. $\phi^\alpha$ for $\alpha =1, \ldots, n-1$ for the coordinates on $S_r$ and $\partial_\alpha = \frac{\partial}{\partial \phi^\alpha}$ for their derivatives. Note that $\delta(\partial_\alpha, \partial_\beta)$ is $O(r^2)$.

First, express $\gamma$ and $\tilde \gamma$ in spherical coordinates on $S_r$:
\begin{align}
\gamma_{\alpha \beta} &= g(\partial_{\alpha}, \partial_{\beta}) \nonumber\\
\tilde \gamma_{\alpha \beta} &= \tilde g(\partial_{\alpha}, \partial_{\beta}) = \gamma_{\alpha \beta} + h_{\alpha \beta}, \label{eqn_tilde_gamma}
\end{align}
where $h_{\alpha \beta} = h(\partial_{\alpha}, \partial_{\beta})$ is $O(r^{3-n})$ by \eqref{decay2}. Both $\gamma_{\alpha \beta}$ and $\tilde \gamma_{\alpha \beta}$ are $O(r^2)$.
Also, we have the inverse metrics:
\begin{align}
\gamma^{\alpha \beta} &= O(r^{-2}) \label{gammaupper}\\
\tilde \gamma^{\alpha \beta} &= \gamma^{\alpha \beta} + O(r^{-n-1}). \label{tildegammaupper}
\end{align}
Note that the derivatives tangent to $S_r$ satisfy:
\begin{align}
\partial_{\mu} \gamma_{\alpha\beta} &= O(r^2) \label{gammaderiv}\\
\partial_{\mu} h_{\alpha\beta} &= O(r^{3-n})\nonumber\\
\partial_{\mu} \tilde \gamma_{\alpha\beta} &= O(r^2), \label{tildegammaderiv}
\end{align}
with the same orders for second derivatives. Similarly,
\begin{align}
\partial_{\mu} \gamma^{\alpha\beta} &= O(r^{-2})
\label{eqn_deriv_inv_metric}\\
\partial_{\mu} \tilde \gamma^{\alpha\beta} &= O(r^{-2}). \label{tildegammaupperderiv}
\end{align}

Next, let $\Gamma$ and $\tilde \Gamma$ denote the Christoffel symbols of $(S_r, \gamma)$, and $(S_r, \tilde \gamma)$, respectively, and define
$$\Psi_{\alpha \beta}^\mu = \tilde \Gamma_{\alpha \beta}^\mu - \Gamma_{\alpha \beta}^\mu.$$
By \eqref{gammaupper}, and \eqref{gammaderiv}, we have
\begin{equation}
\Gamma_{\alpha \beta}^\mu = O(1). \label{eqn_Gamma}
\end{equation}
Using \eqref{eqn_deriv_inv_metric} as well,
\begin{equation}
\partial_\nu \Gamma_{\alpha \beta}^\mu = O(1).\label{eqn_deriv_Gamma}
\end{equation}
Next, we need decay on $\Psi$ and $\partial \Psi$. Using \eqref{eqn_tilde_gamma}--\eqref{tildegammaupper},
\begin{align*}
\Psi_{\alpha \beta}^\mu 
&=\tilde \gamma^{\mu \nu} \left( \partial_\beta \tilde \gamma_{\alpha \nu} + \partial_\alpha \tilde \gamma_{\beta \nu} - \partial_\nu \tilde \gamma_{\alpha \beta}\right) - \gamma^{\mu \nu}\left( \partial_\beta  \gamma_{\alpha \nu} + \partial_\alpha  \gamma_{\beta \nu} - \partial_\nu  \gamma_{\alpha \beta}\right) \\
&= O(r^{-2}) \left( \partial_\beta  h_{\alpha \nu} + \partial_\alpha  h_{\beta \nu} - \partial_\nu  h_{\alpha \beta}\right) + O(r^{-n-1}) \left( \partial_\beta \tilde \gamma_{\alpha \nu} + \partial_\alpha \tilde \gamma_{\beta \nu} - \partial_\nu \tilde \gamma_{\alpha \beta}\right).
\end{align*}
Since $h_{\alpha \beta}$ and $\partial_\mu h_{\alpha\beta}$ are $O(r^{3-n})$, and also by \eqref{tildegammaderiv}, we have
$$\Psi_{\alpha \beta}^\mu = O(r^{1-n}),$$
and a similar calculation, using \eqref{tildegammaupperderiv}, shows
$$\partial_\nu \Psi_{\alpha \beta}^\mu = O(r^{1-n}).$$

Finally:
\begin{align*}
\tilde \rho &= \tilde \gamma^{\beta \mu}\left(\partial_\alpha \tilde \Gamma_{\beta \mu}^\alpha - \partial_\mu \tilde \Gamma_{\alpha\beta}^\alpha + \tilde\Gamma_{\beta \mu}^\nu\tilde\Gamma_{\alpha \nu}^\alpha - \tilde\Gamma_{\alpha\beta}^\nu\tilde\Gamma_{\mu\nu}^\alpha\right)\\
&= \left( \gamma^{\beta \mu} + O(r^{-n-1})\right)\left[\partial_\alpha \left(\Gamma_{\beta \mu}^\alpha +\Psi_{\beta \mu}^\alpha\right) - \partial_\mu \left( \Gamma_{\alpha\beta}^\alpha + \Psi_{\alpha\beta}^\alpha\right)\right.\\
&\qquad\qquad \left.+ \left(\Gamma_{\beta \mu}^\nu +\Psi_{\beta \mu}^\nu\right) \left(\Gamma_{\alpha \nu}^\alpha+\Psi_{\alpha \nu}^\alpha\right) - \left(\Gamma_{\alpha\beta}^\nu+\Psi_{\alpha\beta}^\nu\right)\left(\Gamma_{\mu\nu}^\alpha+\Psi_{\mu\nu}^\alpha\right)\right]\\
&= \rho + O(r^{-n-1}),
\end{align*}
having used \eqref{gammaupper}, \eqref{tildegammaupper}, \eqref{eqn_Gamma}, and \eqref{eqn_deriv_Gamma}. Combining this with \eqref{eqn_rho}, equation \eqref{eqn_tilde_rho} follows.
 
\medskip

For the second part of the proof, we must compute the mean curvature of large coordinate spheres $S_r$ with respect to $\tilde g$. We approach this through the first variation of area.
Let $\omega_0, \omega$, and $\tilde \omega$ be the area forms of $S_r$ induced by $\delta$, $g$ and $\tilde g$, respectively. The respective mean curvature vectors $\mathbf {H_0}, \mathbf {H},\mathbf {\tilde H}$ of $S_r$ with respect to these metrics are characterized by the first variation of area formulas as follows:
\begin{align}
D_X \omega_0 &= \delta(X,  -\mathbf {H_0}) \omega_0  = \delta\left(X,  \frac{n-1}{r}\cdot \partial_r\right) \omega_0 \label{FVOA1}\\
D_X \omega &= g(X,  -\mathbf H) \omega\label{FVOA2}\\
D_X \tilde \omega &= \tilde g(X,  - \mathbf {\tilde H} ) \tilde \omega,\label{FVOA3}
\end{align}
where $D_X$ denotes an infinitesimal deformation of $S_r$ in the direction of $X$, where $X$ is a tangent vector field to $M$ along $S_r$.

We again use spherical coordinates as in the first part of the proof. Note that $(\phi^\alpha)$ give coordinates on $S_r$ that are orthogonal with respect to $\delta$, and hence with respect to the conformal metric $g$. In addition to the estimates of $\gamma_{\alpha \beta}, h_{\alpha \beta}$ and their tangential derivatives used in the first part of the proof, we also need estimates on the radial derivatives. By the decay of $g$ and $h$, as well as by \eqref{gammaupper}, we obtain:
\begin{align*}
\partial_{r} \gamma_{\alpha\beta} &= O(r^1)\\
\partial_{r} \gamma^{\alpha\beta} &= O(r^{-3})\\
\partial_{r} h_{\alpha\beta} &= O(r^{2-n}).
\end{align*}

We begin by computing the mean curvature $H$ of $S_r$ with respect to $g$; this is well-known, but we include it for completeness. The area forms $\omega_0$ and $\omega$ on $S_r$ are related by
\begin{align*}
\omega &= \left(1+ \frac{m}{2r^{n-2}}\right)^{\frac{2(n-1)}{n-2}} \omega_0.
\end{align*}
Then, using \eqref{FVOA1}, elementary calculations show:
\begin{align}
D_r \omega &=\frac{2(n-1)}{n-2} \left(1+ \frac{m}{2r^{n-2}}\right)^{\frac{2(n-1)}{n-2}-1}\cdot \left(\frac{(2-n)m}{2r^{n-1}}\right) \omega_0 +  \left(1+ \frac{m}{2r^{n-2}}\right)^{\frac{2(n-1)}{n-2}} D_r \omega_0\nonumber\\
&=\left(1+ \frac{m}{2r^{n-2}}\right)^{\frac{2(n-1)}{n-2}}\cdot\left[\frac{n-1}{r} - \frac{m(n-1)}{r^{n-1}} \left(1+ \frac{m}{2r^{n-2}}\right)^{-1}\right] \omega_0, \label{Dromega1}
\end{align}
where $D_r =D_{\partial_r}$. Now, using \eqref{FVOA2}, we have: 
\begin{align}
D_r \omega &= g(\partial_r,  -\mathbf H) \omega \nonumber\\
&=\left(1+ \frac{m}{2r^{n-2}}\right)^{\frac{2n}{n-2}} H \omega_0,\label{Dromega2}
\end{align}
where $H = |\mathbf{H}|_g$.
Now, combining \eqref{Dromega1} and \eqref{Dromega2}, elementary calculations show
\begin{equation}
H = \frac{n-1}{r} - \frac{(n-1)^2 m}{(n-2)r^{n-1}} + O(r^{-n}). \label{H_g}
\end{equation}

Now, we proceed to estimate the mean curvature with respect to $\tilde g$. Define a function $\Phi>0$ on the asymptotically flat end of $M$ so that
\begin{equation}
\label{eqn_Phi}
\tilde \omega = \sqrt{\Phi} \omega \qquad \text{ on } S_r
\end{equation}
i.e.,
$$\Phi = \frac{\det(\tilde \gamma_{\alpha \beta})}{\det( \gamma_{\alpha \beta})}.$$
Using Jacobi's formula for the derivative of the determinant, along with the known decay of $\gamma_{\alpha \beta},$ and $\tilde \gamma_{\alpha\beta}$  and their derivatives, we have the following asymptotics of $\Phi$:
\begin{align}
\Phi &= 1 + O(r^{1-n}) \label{Phi1}\\
\partial_\mu \Phi &= O(r^{1-n}) \label{Phi2}\\
\partial_r \Phi &= O(r^{-n}). \label{Phi3}
\end{align}

In order to compute $\mathbf{\tilde H}$, we compute tangential and radial variations of $\tilde \omega$ beginning with \eqref{eqn_Phi}:
\begin{align}
D_\mu \tilde \omega &= \frac{1}{2}\left( \partial_\mu \Phi\right) \Phi^{-1/2} \omega + \sqrt{\Phi}D_\mu \omega \nonumber\\
&= \frac{1}{2}\left( \partial_\mu \Phi\right) \Phi^{-1/2} \omega + \sqrt{\Phi} g(\partial_\mu, -\mathbf{H}) \omega \nonumber\\
&= O(r^{1-n}) \omega, \label{D_mu_tilde_omega}
\end{align}
where we have used the fact that $\mathbf H$ is $g$-orthogonal to $S_r$, as well as \eqref{FVOA2} and \eqref{Phi1}--\eqref{Phi2}. Next, for the radial directions:
\begin{align}
D_r \tilde \omega 
&= \frac{1}{2}\left( \partial_r \Phi\right) \Phi^{-1/2} \omega + \sqrt{\Phi} g(\partial_r, -\mathbf{H}) \omega \nonumber\\
&= g(\partial_r, -\mathbf{H})\omega + O(r^{-n}) \omega, \label{D_r_tilde_omega}
\end{align}
having used \eqref{Phi1}, \eqref{Phi3}, and $H=O(r^{-1})$. 
The goal is to combine the last two statements with \eqref{FVOA3}. Specifically, we estimate \eqref{FVOA3} as follows:
\begin{align*}
D_X \tilde \omega &= (g+h)(X,  - \mathbf {\tilde H} ) \sqrt{\Phi} \omega\\
&=  g(X,  - \mathbf {\tilde H} ) \omega +  O (r^{-n})|X|_g \omega,
\end{align*}
having used the decay of $h$, \eqref{Phi1}, and $|\mathbf{\tilde H}|_g = O(r^{-1})$.
Define $\mathbf{Y}= \mathbf{\tilde H} - \mathbf{H}$. Then applying \eqref{D_r_tilde_omega} and the last equation (with $X = \partial_r$) and applying \eqref{D_mu_tilde_omega} and the last equation (with $X = \partial_\mu$) produces:
\begin{align}
g(\partial_r, \mathbf{Y}) &= O(r^{-n}) \label{g_d_r}\\
g(\partial_\mu, \mathbf{Y}) &= O(r^{1-n}). \label{g_d_mu}
\end{align}
By expanding $|\mathbf{Y}|_g^2$ in the $g$-orthogonal basis $(\partial_r, \partial_\mu)$ of $TM$ along $S_r$, and using \eqref{g_d_r}--\eqref{g_d_mu}, we obtain
\begin{equation}
|\mathbf{Y}|^2_g = O(r^{-2n}). \label{Y_norm}
\end{equation}
Finally, letting $\tilde H = |\mathbf{\tilde H}|_{\tilde g}$, we use the triangle inequality to show:
\begin{align*}
\left|\tilde H - H\right| 
&\leq \Big| |\mathbf{\tilde H}|_{\tilde g} - |\mathbf{H}|_{\tilde g}\Big| + \Big| |\mathbf{H}|_{\tilde g}-|\mathbf H|_g\Big|\\
&\leq |\mathbf{Y}|_{\tilde g} + \Big| |\mathbf{H}|_{\tilde g}-|\mathbf H|_g\Big|\\
&=\Big(g(\mathbf{Y},\mathbf{Y})+h(\mathbf{Y},\mathbf{Y})\Big)^{\frac{1}{2}} + \left|\left(g(\mathbf{H},\mathbf{H}) + h(\mathbf{H},\mathbf{H})\right)^{\frac{1}{2}} - g(\mathbf{H},\mathbf{H})^{\frac{1}{2}} \right|\\
&=|\mathbf{Y}|_g +|\mathbf{Y}|_g O(r^{1-n}) + H\cdot O(r^{1-n})\\
&= O(r^{-n}),
\end{align*}
by \eqref{Y_norm}.
Combining this with \eqref{H_g}, \eqref{eqn_tilde_H} follows.
\end{proof}


\begin{bibdiv}
 \begin{biblist}

 \bib{ADM}{article}{
   author={Arnowitt, R.},
   author={Deser, S.},
   author={Misner, C.},
   title={Coordinate invariance and energy expressions in general relativity},
   journal={Phys. Rev. (2)},
   volume={122},
   date={1961},
   pages={997--1006},
}

\bib{Ba0}{article}{
   author={Bartnik, R.},
   title={The mass of an asymptotically flat manifold},
   journal={Comm. Pure Appl. Math.},
   volume={39},
   date={1986},
   number={5},
   pages={661--693},
}

\bib{Ba1}{article}{
   author={Bartnik, R.},
   title={New definition of quasilocal mass},
   journal={Phys. Rev. Lett.},
   volume={62},
   date={1989},
   number={20},
   pages={2346--2348}
}

\bib{Ba2}{article}{
   author={Bartnik, R.},
   title={Energy in general relativity},
   conference={
      title={Tsing Hua lectures on geometry \&\ analysis},
      address={Hsinchu},
      date={1990--1991},
   },
   book={
      publisher={Int. Press, Cambridge, MA},
   },
   date={1997},
   pages={5--27}
}

\bib{Ba3}{article}{
   author={Bartnik, R.},
   title={Mass and 3-metrics of non-negative scalar curvature},
   conference={
      title={Proceedings of the International Congress of Mathematicians,
      Vol. II},
   },
   book={
      publisher={Higher Ed. Press, Beijing},
   },
   date={2002},
   pages={231--240}
}



\bib{B}{article}{
   author={Bray, H.},
   title={Proof of the Riemannian Penrose inequality using the positive mass
   theorem},
   journal={J. Differential Geom.},
   volume={59},
   date={2001},
   number={2},
   pages={177--267}
}

\bib{BL}{article}{
   author={Bray, H.},
   author={Lee, D.},
   title={On the Riemannian Penrose inequality in dimensions less than
   eight},
   journal={Duke Math. J.},
   volume={148},
   date={2009},
   number={1},
   pages={81--106}
}
	
	
\bib{CS}{article}{
  author={Carlotto, A.},
  author={Schoen, R.},
  title={Localizing solutions of the Einstein constraint equations},
  journal={Invent. Math.},
  volume={205},
  pages={559--615},
  date={2015}
  }

\bib{Chr}{article}{
   author={Chru\'sciel, P.},
   title={Boundary conditions at spatial infinity from a Hamiltonian point
   of view},
   conference={
      title={Topological properties and global structure of space-time},
      address={Erice},
      date={1985},
   },
   book={
      series={NATO Adv. Sci. Inst. Ser. B Phys.},
      volume={138},
      publisher={Plenum, New York},
   },
   date={1986},
   pages={49--59}
}



\bib{DM}{article}{
   author={Dai, X.},
   author={Ma, L.},
   title={Mass under the Ricci flow},
   journal={Comm. Math. Phys.},
   volume={274},
   date={2007},
   number={1},
   pages={65--80}
}

\bib{FST}{article}{
   author={Fan, X.-Q.},
   author={Shi, Y.},
   author={Tam, L.-F.},
   title={Large-sphere and small-sphere limits of the Brown-York mass},
   journal={Comm. Anal. Geom.},
   volume={17},
   date={2009},
   number={1},
   pages={37--72},
}

\bib{J}{article}{
   author={Jauregui, J.},
   title={On the lower semicontinuity of the ADM mass},
   journal={Comm. Anal. Geom},
   volume={26},
   number={1},
   date={2018}
}


\bib{JL}{article}{
   author={Jauregui, J.},
   author={Lee, D.},
   title={Lower semicontinuity of mass under $C^0$ convergence and Huisken's isoperimetric mass},
   journal={J. Reine Angew. Math., to appear}
}


\bib{HI}{article}{
   author={Huisken, G.},
   author={Ilmanen, T.},
   title={The inverse mean curvature flow and the Riemannian Penrose
   inequality},
   journal={J. Differential Geom.},
   volume={59},
   date={2001},
   number={3},
   pages={353--437},
}



\bib{IMS}{article}{
   author={Isenberg, J.},
   author={Mazzeo, R.},
   author={Sesum, N.},
   title={Ricci flow on asymptotically conical surfaces with nontrivial
   topology},
   journal={J. Reine Angew. Math.},
   volume={676},
   date={2013},
   pages={227--248}
}

\bib{Li}{article}{
   author={Li, Y.},
   title={Ricci flow on asymptotically Euclidean manifolds},
   journal={Geom. Topol.},
   volume={22},
   date={2018},
   number={3},
   pages={1837--1891}
}

\bib{MM}{article}{
	author={McCormick, S.},
	author={Miao, P.},
	title={On a Penrose-like inequality in dimensions less than eight},
    journal={Int. Math. Res. Not., to appear}
}

\bib{MTX}{article}{
   author={Miao, P.},
   author={Tam, L.-F.},
   author={Xie, N.},
   title={Quasi-local mass integrals and the total mass},
   journal={J. Geom. Anal.},
   volume={27},
   date={2017},
   number={2},
   pages={1323--1354},
   issn={1050-6926},
   review={\MR{3625154}},
   doi={10.1007/s12220-016-9721-z},
}

\bib{OW}{article}{
   author={Oliynyk, T.},
   author={Woolgar, E.},
   title={Rotationally symmetric Ricci flow on asymptotically flat
   manifolds},
   journal={Comm. Anal. Geom.},
   volume={15},
   date={2007},
   number={3},
   pages={535--568},
}

\bib{Sch}{article}{
   author={Schoen, R.},
   title={Variational theory for the total scalar curvature functional for
   Riemannian metrics and related topics},
   conference={
      title={Topics in calculus of variations},
      address={Montecatini Terme},
      date={1987},
   },
   book={
      series={Lecture Notes in Math.},
      volume={1365},
      publisher={Springer, Berlin},
   },
   date={1989},
   pages={120--154}
}
	

\bib{SY}{article}{
	author={Schoen, R.},
	author={Yau, S.-T.},
	title={On the proof of the positive mass conjecture in general relativity},
	journal={Comm. Math. Phys.},
	volume={65},
	year={1979},
	pages={45--76},
}

\bib{SY2}{article}{
   author={Schoen, R.},
   author={Yau, S.-T.},
   title={Complete manifolds with nonnegative scalar curvature and the
   positive action conjecture in general relativity},
   journal={Proc. Nat. Acad. Sci. U.S.A.},
   volume={76},
   date={1979},
   number={3},
   pages={1024--1025}
}


\bib{SW}{article}{
   author={Sormani, C.},
   author={Wenger, S.},
   title={The intrinsic flat distance between Riemannian manifolds and other
   integral current spaces},
   journal={J. Differential Geom.},
   volume={87},
   date={2011},
   number={1},
   pages={117--199}
}

\bib{Wh}{article}{
   author={White, B.},
   title={The size of the singular set in mean curvature flow of mean-convex
   sets},
   journal={J. Amer. Math. Soc.},
   volume={13},
   date={2000},
   number={3},
   pages={665--695 (electronic)}
}


\bib{W}{article}{
	author={Witten, E.},
	title={A new proof of the positive energy theorem},
	journal={Comm. Math. Phys.},
	volume={80},
	year={1981},
	pages={381-402},
}	

 \end{biblist}
\end{bibdiv}
\end{document}